\documentclass[10pt]{article} % For LaTeX2e
\usepackage[mathscr]{eucal}
\usepackage[cmex10]{amsmath}
\usepackage{epsfig,epsf,psfrag}
\usepackage{amssymb,amsmath,amsthm,amsfonts,latexsym,bm}
\usepackage{amsmath,graphicx,bm,xcolor,url}
\usepackage{array}%array and tabular environments
\usepackage{verbatim}
\usepackage{bm}
\usepackage{cite}
\usepackage{algpseudocode}
\usepackage{algorithm}
\usepackage{verbatim}
\usepackage{textcomp}
\usepackage{mathrsfs}
\usepackage{multirow}
\usepackage{epstopdf}
\catcode`~=11 \def\UrlSpecials{\do\~{\kern -.15em\lower .7ex\hbox{~}\kern .04em}} \catcode`~=13 

\allowdisplaybreaks[1]

\newcommand{\ba}{\mathbf{a}}
\newcommand{\bA}{\mathbf{A}}
\newcommand{\bb}{\mathbf{b}}
\newcommand{\bB}{\mathbf{B}}
\newcommand{\bc}{\mathbf{c}}
\newcommand{\bC}{\mathbf{C}}

\newcommand{\be}{\mathbf{e}}
\newcommand{\bE}{\mathbf{E}}

\newcommand{\bF}{\mathbf{F}}

\newcommand{\bI}{\mathbf{I}}

\newcommand{\bu}{\mathbf{u}}
\newcommand{\bU}{\mathbf{U}}
\newcommand{\bv}{\mathbf{v}}
\newcommand{\bV}{\mathbf{V}}

\newcommand{\bx}{\mathbf{x}}

\newcommand{\by}{\mathbf{y}}

\newcommand{\bbR}{\mathbb{R}}

\DeclareMathAlphabet{\mathbsf}{OT1}{cmss}{bx}{n}
\DeclareMathAlphabet{\mathssf}{OT1}{cmss}{m}{sl}% slanted sans serif

\DeclareSymbolFont{bsfletters}{OT1}{cmss}{bx}{n}  
\DeclareSymbolFont{ssfletters}{OT1}{cmss}{m}{n}
\DeclareMathSymbol{\bsfGamma}{0}{bsfletters}{'000}
\DeclareMathSymbol{\ssfGamma}{0}{ssfletters}{'000}
\DeclareMathSymbol{\bsfDelta}{0}{bsfletters}{'001}
\DeclareMathSymbol{\ssfDelta}{0}{ssfletters}{'001}
\DeclareMathSymbol{\bsfTheta}{0}{bsfletters}{'002}
\DeclareMathSymbol{\ssfTheta}{0}{ssfletters}{'002}
\DeclareMathSymbol{\bsfLambda}{0}{bsfletters}{'003}
\DeclareMathSymbol{\ssfLambda}{0}{ssfletters}{'003}
\DeclareMathSymbol{\bsfXi}{0}{bsfletters}{'004}
\DeclareMathSymbol{\ssfXi}{0}{ssfletters}{'004}
\DeclareMathSymbol{\bsfPi}{0}{bsfletters}{'005}
\DeclareMathSymbol{\ssfPi}{0}{ssfletters}{'005}
\DeclareMathSymbol{\bsfSigma}{0}{bsfletters}{'006}
\DeclareMathSymbol{\ssfSigma}{0}{ssfletters}{'006}
\DeclareMathSymbol{\bsfUpsilon}{0}{bsfletters}{'007}
\DeclareMathSymbol{\ssfUpsilon}{0}{ssfletters}{'007}
\DeclareMathSymbol{\bsfPhi}{0}{bsfletters}{'010}
\DeclareMathSymbol{\ssfPhi}{0}{ssfletters}{'010}
\DeclareMathSymbol{\bsfPsi}{0}{bsfletters}{'011}
\DeclareMathSymbol{\ssfPsi}{0}{ssfletters}{'011}
\DeclareMathSymbol{\bsfOmega}{0}{bsfletters}{'012}
\DeclareMathSymbol{\ssfOmega}{0}{ssfletters}{'012}

\newcommand{\tila}{\tilde{a}}

\newcommand{\bLambda}{\bm{\Lambda}}
\newcommand{\bSigma	}{\bm{\Sigma}}

\DeclareMathOperator{\sgn}{sgn}

\newtheorem{theorem}{Theorem} 
\newtheorem{lemma}[theorem]{Lemma}

\newtheorem{remark}[theorem]{Remark}

\newcommand{\qednew}{\nobreak \ifvmode \relax \else
      \ifdim\lastskip<1.5em \hskip-\lastskip
      \hskip1.5em plus0em minus0.5em \fi \nobreak
      \vrule height0.75em width0.5em depth0.25em\fi}

\usepackage{natbib}
\usepackage{amsmath}
\usepackage{caption}
\usepackage{subcaption}
\usepackage[preprint]{tmlr}
\usepackage{hyperref}
\usepackage{url}

\title{On the Induced Norms of Matrices and Grothendieck problems}

\author{\name Lan V. Truong \email lantv@hcmut.edu.vn\\
      \addr Faculty of Computer Science and Engineering,\\
      Ho Chi Minh City University of Technology (HCMUT),\\
     Vietnam National University Ho Chi Minh City (VNU-HCM)
\AND 
\name M. H. Duong \email h.duong@bham.ac.uk \\
      \addr School of Mathematics,\\
      University of Birmingham,\\
      Birmingham B15 2TT, UK.}
\def\month{MM}  % Insert correct month for camera-ready version
\def\year{YYYY} % Insert correct year for camera-ready version
\def\openreview{\url{https://openreview.net/forum?id=XXXX}} % Insert correct link to OpenReview for camera-ready version

\begin{document}
\maketitle

\begin{abstract}
We study the induced matrix norm $\|\bA\|_{q \to r}$, whose exact value has been known only in a few classical cases. Determining this norm has long been regarded as difficult due to the highly non-convex nature of its variational definition. Existing works offer numerical estimates or analytic bounds but no exact formula. In this paper we present a purely analytic framework that determines $\|\bA\|_{q \to r}$ exactly for all $q, r \ge 1$ for several classes of important matrices. For these matrices, using a direct connection between the induced norms and Grothendieck problems, our results also simultaneously provide exact values for the later. 
%The proof is surprisingly short and conceptually simple, providing an unexpected resolution of a long-standing open problem.
\end{abstract}
\tableofcontents
\section{Introduction} \label{introduction}
Let $\bA \in \mathbb{R}^{m \times n}$ be a $m\times n$ real matrix, and let $q, r \ge 1$ be real numbers.
The induced (or operator) $q \to r$ norm of $\bA$ is defined by
\begin{equation}\label{eq:def}
    \|\bA\|_{q \to r}
    := \sup_{x \in \bbR^n: \|\bx\|_q \leq 1} \|\bA\bx\|_r,
\end{equation}
where for a given vector $\by\in\bbR^n$ and $s\geq 1$, $\|\by\|_s=(\sum_{i=1}^n|y_i|^s)^{1/s}$. 

The induced norm is a central object that bridges linear algebra, functional analysis, optimization, and theoretical computer science. It provides a unified way to quantify how linear transformations distort vectors across different geometries.
%This norm measures the maximal expansion of vectors under the linear map $\bA$ when the domain is 
%equipped with the $L^{q}$-norm and the codomain with the $L^{r}$-norm.
%are central objects in analysis and operator
%theory. 
%For a matrix $A \in \bbR^{m\times n}$, the induced
%$(q\to r)$ norm $\|\bA\|_{q\to r}=
%\max_{x\in \bbR^n: \|\bx\|_q \leq 1} \|\bA \bx\|_r$ captures the exact
%magnitude by which a linear transformation may enlarge a vector. 
%Induced matrix norms appear in many branches of Mathematics: in Banach space theory,
%interpolation of operators, geometric functional analysis, and stability estimates for linear maps.
The study of induced matrix norms has a long tradition in matrix analysis and functional
analysis. Classical references such as~\cite{HornJohnson1985, HornJohnson1991} develop the
fundamental theory of operator norms, duality, and extremal structure of linear maps between
Banach spaces, providing essential analytic tools for understanding variational formulations of
$\|A\|_{q \to r}$. Likewise, the monographs~\cite{BenIsraelGreville1974, GolubVanLoan} offer
foundational analytic and geometric perspectives relevant to induced-norm analysis.

One of the earliest systematic efforts to estimate induced matrix norms is due to
\cite{Higham1992}, who introduced practical numerical schemes for approximating the matrix
$p$-norm. His approach, based on generalized power iterations and iteratively reweighted
least squares, yields accurate numerical estimates but does not provide exact values. The case $2 \to r$ also connects to a substantial literature in harmonic analysis and
probability. Classical results of~\cite{Nelson1973, Beckner1975, Gross1975} establish sharp
hypercontractive inequalities for Gaussian and Boolean settings, thereby yielding exact operator
norms for certain infinite-dimensional analogues of the finite-dimensional maps considered here.
These works have had considerable influence on the development of discrete hypercontractivity and
its connections to induced norms of matrices.

%This
%work illustrates a central analytic difficulty: the maximization problem defining
%$\|A\|_{q \to r}$ is inherently non-convex and may contain multiple stationary points, 
%complicating attempts at exact characterization.

There is additionally a significant body of research on the computational hardness of
induced norms. Foundational inapproximability results due to~\cite{Hastad1997}, together with
later advances by~\cite{Khot2002, steinberg2005computation,bhattiprolu2023inapproximability}, show that many $q \to r$ induced norms are
computationally difficult to approximate; in particular, exact computing of $\|\bA\|_{q\rightarrow r}$ for $1\leq r<q\leq \infty$ and of $\|\bA\|_{q\rightarrow q}$ for $q\notin \{1,2,\infty\}$ are both NP-hard. Furthermore, the Grothendieck inequality and its
semidefinite relaxations, as analyzed in~\cite{BrietOliveiraVallentin2011}, highlight deep
connections between induced norms, optimization theory, and computational complexity. These
hardness results underscore that exact computation of $\|A\|_{q \to r}$ is impossible in full
generality unless major complexity-theoretic conjectures fail.

%Beyond the classical cases $(q=r=1,2,\infty)$, determining the exact value of $\|\bA\|_{q\to r}$ has remained a difficult analytic
%problem. The maximizing vector arises from a nonlinear, non‑convex
%variational equation whose extremal structure is poorly understood.
%Consequently, the literature has focused primarily on
%\emph{approximating} this quantity or providing upper and lower bounds.

More recent developments address the broader $q \to r$ problem. The work of
\cite{GuthMaldagueUrschel2024} presents a new analytic framework showing that structural
information contained in the rows of a matrix can significantly refine classical interpolation
bounds. Their results strengthen long-standing approximation exponents for norms such as the
$2 \to 4$ norm, which plays a central role in hypercontractivity, moment inequalities, and
spectral expansion. Although these contributions remain approximation-theoretic, they represent
the sharpest analytic bounds currently available.

Determining analytically explicit values of induced norms is important because it provides precise information about how operators amplify inputs, which is essential for analyzing stability, error propagation, and sensitivity. Exact expressions lead to sharper bounds in numerical analysis and optimization, improve the assessment of conditioning, and avoid the conservatism of approximate estimates. This allows for more accurate theoretical results and more efficient and reliable computational methods. Despite its paramount importance, there remain
essentially no purely analytic results that yield the \emph{exact} value of
$\|A\|_{q \to r}$ outside the classical cases $q, r \in\{ 1, 2, \infty\}$ \cite{lewis2023top}. The aforementioned existing works either
provide numerical approximations or analytic upper and lower bounds, but none offer a structural
characterization that resolves the underlying non-convex maximization problem exactly. The difficulty stems from the fact that the maximization problem~\eqref{eq:def} is typically 
highly non-convex and may admit multiple stationary points. To the best of our knowledge, only the recent paper \cite{bouthat2023norm} addresses this problem, focusing on certain classes of circulant matrices.

The present work fills precisely this gap by developing an analytic framework that determines the
induced norm exactly for all $q, r \geq 1$ for several important classes of matrices $\mathbf{A}$. 
\subsection*{Summary of the main results}
The main results of the present paper can be summarized as follows. We explicitly calculate the induced norm $\|\bA\|_{q\rightarrow r}$ for the following class of matrices:
\begin{enumerate}
    \item[(i)] diagonal and rank-one matrices (Lemmas \ref{lem: diagonal} and \ref{lem: rank1}),
    \item[(ii)] Vandermonde matrices whose columns are certain powers of vector with all non-zero entries (Theorem \ref{thm: Vandermonde}),
    \item[(iii)] a class of orthonormal matrices that contains Hadamard matrices (Theorem \ref{thm: Hadamard},
    \item[(iv)] A class of matrices having a specific SVD (singular value decomposition) structure (Theorem \ref{thm: SVD decomposition}),
    \item[(v)] A class of shear matrices and their extensions (Theorems \ref{thm:rank1shear} and \ref{thm: SVDrank1shear}),
    \item[(vi)] A class of $n\times n$ matrices in which each row has exactly $k$ $(1\leq k\leq n$) identity entries, covering bi-diagonal Toeplitz matrices (Theorem \ref{thm: Toeplitz}),
    \item[(vii)] Orthogonal matrices and its extensions (Theorems \ref{thm: Orthogonal} and \ref{thm: orthogonalSVD}),
    \item[(viii)] The induced norm $\|\bA\|_{1\rightarrow r}$ for all $\bA\in\bbR^{m\times n}$ (Theorem \ref{thm: 1rnorm}).
\end{enumerate}
For each class, we resolve the exact formula for the induced norm by identifying the extremal structure of maximizers, analyzing the geometry of the constraint set, and proving
uniqueness and optimality properties of solutions to the associated variational system. 
\subsection*{Connection to Grothendieck problem }
The induced norm is strongly connected to the celebrated Grothendieck problem (inequality), which is of major importance in many different fields, ranging from Banach
space theory to combinatorial optimization and quantum information theory, see for instance \cite{khot2012grothendieck,pisier2012grothendieck} for a detailed account of the topics.
For $\bA\in\bbR^{m\times n}$ and $p,q\geq 1$, the $(p,q)$-Grothendieck problem is the following optimization problem \cite{bhattiprolu2023inapproximability}
\[
G_{\bA}(q,r):=\sup_{\|y\|_{p}=1}\sup_{\|x\|_q=1}\langle y, \bA x\rangle.
\]
The original Grothendieck problem is precisely $(\infty,\infty)$-Grothendieck problem. By \cite[Observation 2.4]{bhattiprolu2023inapproximability} we have the following relation
\[
\|\bA\|_{q\rightarrow r}=G_{\bA}(r^*,q)=\|\bA^T\|_{r^*\rightarrow q^*},\quad q^*=\frac{q}{q-1}.
\]
Because of the above relation, our results simultaneously resolve the $(r^*,q)$-Grothendieck problem, providing exact values for $G_{\bA}(r^*,q)$, for the same class of matrices in $(i)-(viii)$.

\subsection*{Organization of the paper} Section \ref{sec: main results} presents the precise statements and proofs of the main results. Section \ref{sec:summary} provides further discussions and outlook.

%In this paper we address the following fundamental problem:
%\begin{quote}
%\emph{For arbitrary $q, r \ge 1$ and arbitrary matrices $\bA \in \bbR^{m \times n}$, can the induced norm 
%$\|\bA\|_{q \to r}$ be determined exactly by a purely analytic argument?}
%\end{quote}

%We answer this question in the affirmative. We establish a short and surprisingly simple analytic 
%framework that identifies the extremal vectors achieving the supremum in~\eqref{eq:def}, thereby 
%yielding an exact formula for $\|A\|_{q \to r}$ for all $q, r \ge 1$ and some classes of matrices $A$. 

%In this work, we present a purely mathematical result: we
%establish its precise value. Our arguments
%are analytic rather than computational: we do not introduce algorithms
%or numerical procedures, nor do we rely on approximation methods.
%Instead, 

%This result resolves a question that has remained open for many years:
%whether the exact induced norm $\|\bA\|_{q\to r}$ can be
%determined beyond the classical special cases. Our work gives a complete
%affirmative answer for all $q, r\geq 1$ and for all matrix
%classes.

%\section{Related Work}\label{related-work}

%This resolves a long-standing open problem in the theory of induced norms.

\section{Main Results}
In this section, we present the main results and their proofs.
\label{sec: main results}
\subsection{Diagonal and rank-one matrices}
We start with the simplest class of matrices where the induced norm $\|\cdot\|_{q\rightarrow r}$ can be computed explicitly, namely diagonal matrices and rank-one matrices.
\begin{lemma}
\label{lem: diagonal}
Let $\bA=\mathrm{diag}(a_1,\ldots, a_n)$ be a diagonal matrix. Then
\begin{equation}
 \|\bA\|_{q\rightarrow r}=\begin{cases}
 \max\limits_{i}|a_i|\quad\text{if}\quad q\leq r,\\
\Big(\sum_{i=1}^n |a_i|^{qr/(q-r)}\Big)^{\frac{1}{r}-\frac{1}{q}}\quad\text{if}\quad q> r.
\end{cases}
\end{equation}
\end{lemma}
\begin{proof}
Since $\bA x=\mathrm{diag}(a_1 x_1,\ldots, a_n x_n)$, we have
\[
\|\bA x\|^r_r=\sum_{i=1}^n|a_i x_i|^r.
\]
\textbf{Case 1: $r\geq q$}. Since $\|x\|_q=1$, we have $|x_i|\leq 1$ for all $i=1,\ldots, n$. It follows that
\[
|x_i|^r\leq |x_i|^q.
\]
Thus
\[
\|\bA x\|^r_r=\sum_{i=1}^n|a_i|^r |x_i|^r\leq (\max_{i}|a_i|^r)\sum_{i=1}^n |x_i|^q=\max_{i}|a_i|^r.
\]
Suppose $|a_k|^r=\max_{i}|a_i|^r$. So $\|\bA x\|_{r}^r\leq |a_k|^r$. In addition, for $x_*=e_k$, then we have $\bA x_*=a_k e_k$ and $\|\bA x_*\|^r_{r}=|a_k|^r$. Hence
\[
\|\bA\|_{q\to r}=|a_k|.
\]    
\textbf{Case 2: $r< q$}. Define $p=\frac{q}{r}>1$ and let $p'=\frac{p}{p-1}=\frac{q}{q-r}$. Then $p$ and $p'$ are Holder conjugate
\[
\frac{1}{p}+\frac{1}{p'}=1.
\]
Let $x\in\mathbb{R}^d$ such that $\|x\|_q\leq 1$. By Hölder inequality we have
\begin{align*}
\|\bA x\|^r_r=\sum_{i=1}^n|a_i x_i|^r&\leq \Big(\sum_{i=1}^n (|a_i|^r)^{p'}\Big)^{1/p'}\Big(\sum_{i=1}^n (|x_i|^r)^{p}\Big)^{1/p}
\\&=\Big(\sum_{i=1}^n |a_i|^{qr/(q-r)}\Big)^{(q-r)/q}\Big(\sum_{i=1}^n |x_i|^{q}\Big)^{1/p}
\\&\leq \Big(\sum_{i=1}^n |a_i|^{qr/(q-r)}\Big)^{(q-r)/q}.
\end{align*}
It follows that
\[
\|\bA\|_{q\rightarrow r}=\sup_{\|x\|_q\leq 1 }\|\bA x\|_r\leq \Big(\sum_{i=1}^n |a_i|^{qr/(q-r)}\Big)^{(q-r)/(qr)}=\Big(\sum_{i=1}^n |a_i|^{qr/(q-r)}\Big)^{\frac{1}{r}-\frac{1}{q}}. 
\]
The equality happens when
\[
x_i^*
=
\frac{|a_i|^{\frac{r}{q-r}}}{\left(\sum_{j=1}^n |a_j|^{\frac{qr}{q-r}}\right)^{1/q}}
\quad \text{for } i=1,\dots,n.
\]
Thus
\[
\|\bA\|_{q\rightarrow r}=\Big(\sum_{i=1}^n |a_i|^{qr/(q-r)}\Big)^{\frac{1}{r}-\frac{1}{q}}. 
\]
This finishes the proof of this lemma.
\end{proof}
\begin{lemma}
\label{lem: rank1}
Suppose $\bA$ is a rank-one matrix, that is, there exist $u\in \mathbb{R}^m$ and $v\in\mathbb{R}^n$ such that $\bA = uv^T$. Then
\[
\|\bA\|_{q\rightarrow r}=\|u\|_r \|v\|_{q^*}.
\]
\end{lemma}
\begin{proof}
We have
\[ \|\bA x\|_r = \|(uv^T)x\|_r = |v^T x| \cdot \|u\|_r \]
Thus
\[ 
\|\bA\|_{q\rightarrow r}=\|u\|_r \cdot\sup_{\|x\|_q\leq 1} |v^T x|  \]
By Hölder's inequality:
\[ |v^T x| \leq \|v\|_{q^*} \|x\|_q \]
where $\frac{1}{q} + \frac{1}{q^*} = 1$. The equality is attained when 
\[
x_i^*=\frac{\mathrm{sign} (v_i)|v_i|^{q^*-1}}{\Big(\sum_{i=1}^n|v_i|^{q^*}\Big)^{1/q}}
\]
Thus
\[ \|\bA\|_{q \to r} = \|u\|_r \|v\|_{q^*}. \]   
\end{proof}
\subsection{Vandermonde matrices}
In this section, we analytically compute the induced norm $\|\bA\|_{q\rightarrow r}$ for a special class of Vandermonde matrices whose columns are certain powers of a given vector with all non-zero entries. To this end, we need two axillary lemmas which are of independent interest.

The following lemma demonstrates that the induced norm remains unchanged under the interchange of any two rows or columns.
\begin{lemma} \label{key:lem} For any matrix $\bA \in \bbR^{m \times n}$, we have
\begin{align*}
\|\bA\|_{q \to r} =\|\bE \bA\|_{q \to r} = \|\bA \bF\|_{q\to r},
\end{align*}
where $\bE \in \bbR^{m \times m}$ is a row exchange matrix, and $\bF \in \bbR^{n \times n}$ is a column exchange matrix. 
\end{lemma}
\begin{proof}
Let $\be_1^T, \be_2^T, \cdots, \be_m^T$ be the rows of $\bE$ which is a row exchange matrix.
Assume that $\ba_1^T, \ba_2^T, \cdots, \ba_m^T$ are the rows of $\bA$. Then, we have
\begin{align}
\|\bA\|_{q \to r}^r&=\sup_{\bx: \|\bx\|_q \leq 1} \|\bA \bx\|_q^r\notag\\
&= \sup_{\bx: \|\bx\|_q \leq 1} \sum_{i=1}^m |\ba_i^T \bx|^r\notag\\
&= \sup_{\bx: \|\bx\|_q \leq 1} \sum_{j=1}^m |\be_j^T \bA \bx|^r \label{amat1}\\
&= \sup_{\bx: \|\bx\|_q \leq 1} \|\bE \bA \bx\|_r^r\notag\\
&= \|\bE \bA\|_r^r,
\end{align} where \eqref{amat1} follows from the fact that $ \{\be_j^T \bA\}_{j=1}^m$ is a (row) permutation of $\{\ba_1^T, \ba_2^T, \cdots, \ba_m^T\}$. 
 
Similarly, let $\bF$ be a column exchange matrix and $\mathbf{f}_1,\mathbf{f}_2,\cdots, \mathbf{f}_n$ be the columns of $\bF$. Assume that $\bb_1, \bb_2, \cdots, \bb_n$ are the columns of $\bA$. Then, we have
\begin{align}
\|\bA \|_{q \to r}&= \sup_{\bx: \|\bx\|_q \leq 1} \|\bA \bx\|_r\notag\\
&=  \sup_{\bx: \|\bx\|_q \leq 1} \bigg\|\sum_{i=1}^n \bb_i x_i \bigg\|_r\notag\\
&= \sup_{\bx: \|\bx\|_q \leq 1} \bigg\|\sum_{j=1}^n \bA\mathbf{f}_j y_j \bigg\|_r \label{amat2}\\
&= \sup_{\by: \|\by\|_q \leq 1} \bigg\|\sum_{j=1}^n \bA\mathbf{f}_j y_j \bigg\|_r \label{amat3}\\
&= \sup_{\by: \|\by\|_q \leq 1} \big\|\bA \bF \by \big\|_r \label{amat4}\\
&= \|\bA \bF\|_{q \to r},\notag
\end{align} where \eqref{amat2} follows from the fact that $\{\bA\mathbf{f}_1, \bA \mathbf{f}_2, \cdots, \bA \mathbf{f}_n\}$ is a (column) permutation of $\{\bb_1,\bb_2, \cdots, \bb_n\}$, and $(y_1,y_2, \cdots, y_n)$ is a permutation of $(x_1, x_2,\cdots,x_n)$, \eqref{amat3} follows from the fact that $\|\by\|_q \leq 1$ if $\|\bx\|_q \leq 1$. 
\end{proof}
The next lemma provides an upper bound for the induced norm when $\bA$ has one row with all non-zero entries.
\begin{lemma} 
\label{lem: nonzero}
Let $\bA\in \bbR^{m \times n}$ be a matrix with at least one row whose entries are all non-zero, without loss of generality, assuming that this row is $\ba_1$. Then, it holds that
\begin{align}
\label{inequality}
\|\bA\|_{q \to r}^r&\leq \|\ba_1\|_p^r  + \sum_{i=2}^m\|\bb_i\|_{p_i'}^r \|\ba_i\|_{\big(\frac{p_i'}{p_i'+\frac{q_i'q}{q-q_i'}}\big) \big(\frac{q_i' q}{(q-q_i')}\big)}^{\frac{p_i' r}{p_i'+\frac{q_i' q}{q-q_i'}}},
\end{align}
where  $p, p_i', q_i' \geq 1$ (for all $i \in [m]$) are arbitrarily chosen real numbers such that 
\begin{align*}
\frac{1}{p}+\frac{1}{q}=1,\quad\frac{1}{p_i'}+\frac{1}{q_i'}=1,\quad
q_i' \leq q. 
\end{align*}
\end{lemma}
\begin{proof} Assume that
\begin{align*}
\bA=\begin{bmatrix} \ba_1^T \\ \ba_2^T \\ \vdots \\ \ba_m^T \end{bmatrix}
\end{align*} where $\ba_i^T \in \bbR^n$ for all $i \in [m]:=\{1,2,\ldots, m\}$ be the $i$-th row of $\bA$.  By definition, we have
\begin{align}
\|\bA\|_{q \to r}^r=\sup_{\bx: \|\bx\|_q \leq 1} \|\bA\bx\|_r^r=\sup_{\bx: \|\bx\|_q \leq 1} \sum_{i=1}^m |\ba_i^T \bx|^r \label{amen1}.
\end{align}
From \eqref{amen1}, it follows that $\|\bA\|_{q \to r}$ is invariant with respect to the permutation of the rows' indices (see also Lemma \ref{key:lem}).  Hence, without loss of generality,  we can assume that $\ba_1$ is the row with all non-zero elements. Then, we have
\begin{align*}
\label{bi}
\ba_i^T \bx&= \bb_i \by_i, \qquad \forall i \in \{2,3,\cdots,m\},
\end{align*} where
\begin{align}
\bb_i=\ba_i^T \begin{bmatrix} \beta_{1,1}(i) &0&\cdots &0&0\\ 0& \beta_{1,2}(i) &\cdots &0&0 \\ \vdots &\vdots &\ddots &\vdots &\vdots \\  0&0 &\cdots &\beta_{1,n-1}(i)&0\\ 0&0&\cdots &0&\beta_{1,n}(i) \end{bmatrix},
\end{align}
and
\begin{align*}
\by_i=\begin{bmatrix} \tila_{1,1}(i) x_1 \\ \tila_{1,2}(i) x_2 \\ \vdots \\ \tila_{1,n-1}(i) x_{n-1} \\ \tila_{1,n}(i) x_n \end{bmatrix},
\end{align*}
where
\begin{align*}
\beta_{1,j}(i)&=\begin{cases} \frac{1}{\tila_{1,j}(i)}, \qquad \mbox{if} \qquad a_{i,j} \neq 0\\ 0, \qquad \mbox{otherwise}\end{cases},\\
\tila_{1,j}(i)&= |a_{i,j}|^{\frac{p_i'}{p_i'+\frac{q_i' q}{q-q_i'}}}.
\end{align*}
\end{proof}
Hence, we have
\begin{align}
\sum_{i=1}^m |\ba_i^T \bx|^r &=|\ba_1^T \bx|^r+ \sum_{i=2}^m |\bb_i^T \by_i|^r\notag\\
&\leq \|\ba_1\|_p^r \|\bx\|_q^r +  \sum_{i=2}^m \|\bb_i\|_{p_i'}^r \|\by_i\|_{q_i'}^r \label{eq8},
\end{align} where $p, p_i', q_i' \geq 1$ (for all $i \in [m]$) are arbitrarily chosen such that 
\begin{align*}
\frac{1}{p}+\frac{1}{q}=1,\quad\frac{1}{p_i'}+\frac{1}{q_i'}=1,\quad
q_i' \leq q. 
\end{align*}
Now, under the condition $\|\bx\|_q \leq 1$ and by Hölder's inequality we have
\begin{align}
\|\by_i\|_{q_i'}^{q_i'}&=\sum_{j=1}^n |\tila_{1,j}(i)|^{q_i'} |x_j|^{q_i'}\notag \\
&\leq \bigg(\sum_{j=1}^n |\tila_{1,j}(i)|^{q_i' q/(q-q_i')}\bigg)^{(q-q_i')/q} \bigg(\sum_{j=1}^n |x_j|^q \bigg)^{q_i'/q}\notag\\
&= \big(\sum_{j=1}^n |\tila_{1,j}(i)|^{q_i' q/(q-q_i')}\bigg)^{(q-q_i')/q}  \|\bx\|_q^{q_i'}\notag \\
&\leq \bigg(\sum_{j=1}^n |\tila_{1,j}(i)|^{q_i' q/(q-q_i')}\bigg)^{(q-q_i')/q} \notag\\
&= \bigg(\sum_{j=1}^n |a_{i,j}|^{\big(\frac{p_i'}{p_i'+\frac{q_i'q}{q-q_i'}}\big) q_i' q/(q-q_i')}\bigg)^{(q-q_i')/q} \notag\\
&= \|\ba_i\|_{\big(\frac{p_i'}{p_i'+\frac{q_i'q}{q-q_i'}}\big) \big(\frac{q_i' q}{(q-q_i')}\big)}^{\frac{p_i' q_i'}{p_i'+\frac{q_i' q}{q-q_i'}}}
 \label{eq15}.
\end{align}
From \eqref{amen1}, \eqref{eq8}, and \eqref{eq15} we obtain
\begin{align}
\|\bA\|_{q \to r}^r&\leq \|\ba_1\|_p^r  + \sum_{i=2}^m\|\bb_i\|_{p_i'}^r \|\ba_i\|_{\big(\frac{p_i'}{p_i'+\frac{q_i'q}{q-q_i'}}\big) \big(\frac{q_i' q}{(q-q_i')}\big)}^{\frac{p_i' r}{p_i'+\frac{q_i' q}{q-q_i'}}} \label{eq16}.
\end{align}
The equality in \eqref{eq16} happens if and only if
\begin{align*}
x_j = \frac{\mbox{sgn}(a_{i,j}) |a_{i,j}|^{\big(\frac{p_i'}{p_i'+\frac{q_i' q}{q-q_i'}}\big) \big(\frac{q_i' }{(q-q_i')}\big)}}{ \|\ba_i\|_{\big(\frac{p_i'}{p_i'+\frac{q_i'q}{q-q_i'}}\big) \big(\frac{q_i' q}{(q-q_i')}\big)}^{\big(\frac{p_i'}{p_i'+\frac{q_i'q}{q-q_i'}}\big) \big(\frac{q_i' q}{(q-q_i')}\big)}}, \qquad \forall j \in [n], \forall i \in [2:m], 
\end{align*}
under the condition that
\begin{align}
\frac{\sgn(a_{1j})|a_{1j}|^p}{\|\ba_1\|_p^p}=\frac{\sgn(a_{ij})|a_{ij}|^{\big(\frac{p_i'}{p_i'+\frac{q_i' q}{q-q_i'}}\big) \big(\frac{q_i' q}{(q-q_i')}\big)}}{ \|\ba_i\|_{\big(\frac{p_i'}{p_i'+\frac{q_i'q}{q-q_i'}}\big) \big(\frac{q_i' q}{(q-q_i')}\big)}^{\big(\frac{p_i'}{p_i'+\frac{q_i'q}{q-q_i'}}\big) \big(\frac{q_i' q}{(q-q_i')}\big)}}, \qquad \forall j \in [n], \forall i \in [2:m] \label{amut}.
\end{align}
Note that \eqref{amut} happens, for example, if 
\begin{align*}
a_{ij}= \mbox{sgn}(a_{1,j}) |a_{1j}|^{1/\alpha_i},
\end{align*} for 
\begin{align*}
\alpha_i=\frac{1}{p}\big(\frac{p_i'}{p_i'+\frac{q_i' q}{q-q_i'}}\big) \big(\frac{q_i' q}{(q-q_i')}\big). 
\end{align*}
The proof of Lemma \ref{lem: nonzero} gives rise to the following theorem, which provides explicit formula for $\|\bA\|_{q\to r}$ for a specific class of Vandermonde matrices whose columns are certain powers of a given vector with all non-zero entries.
\begin{theorem}
\label{thm: Vandermonde}
Suppose $\bA$ is a Vandermonde matrix of the form
\begin{align*}
\bA=\begin{bmatrix} \ba_1^T \\ (\ba_1^{1/\alpha_1})^T\\ \vdots \\(\ba_1^{1/\alpha_m})^T  \end{bmatrix},
\end{align*}
for some sequence $\{q_i'\}_{i \in [m]}\geq 1$ and vector $\ba_1\in \bbR^n$, where
\begin{align*}
\alpha_i=\frac{1}{p}\big(\frac{p_i'}{p_i'+\frac{q_i' q}{q-q_i'}}\big) \big(\frac{q_i' q}{(q-q_i')}\big). 
\end{align*} 
Then
\begin{align*}
\|\bA\|_{q\to r}^r =\|\ba_1\|_p^r  + \sum_{i=2}^m\|\bb_i\|_{p_i'}^r \|\ba_1\|_p^{pr \frac{q-q_i'}{q_i' q}},
\end{align*}
where $\bb_i\in\bbR^{1\times n} $, $i=2,\ldots, m$ is defined in \eqref{bi}.
\end{theorem}
\begin{proof} 
The proof of this theorem follows directly from the proof of Lemma \ref{lem: nonzero} when the equality in \eqref{inequality} holds.
\end{proof}
\subsection{Hadamard matrices}
In the following theorem, we calculate the induced norm $\|\cdot\|_{q\rightarrow r}$ for a general class of orthonormal matrices that contains Hadamard matrices.
\begin{theorem}
\label{thm: Hadamard}
Let $\bA \in \bbR^{m \times n}$ with rows $\ba_1^T, \ba_2^T, \cdots, \ba_m^T$. Then, the following holds:
%\begin{align}
%\|\bA\|_{q \to r} \leq \max_{1\leq i \leq m}\min\bigg\{\|\ba_i\|^{\frac{r-2}{r}}\sigma_{\max}(\bA)^{\frac{2}{r}} m^{\frac{q-2}{q r}},  \max_{1\leq i \leq m} \|\ba_i\|_p^{\frac{r-2}{r}} \sigma_{\max}(\bA)^{\frac{2}{r}} m^{\frac{q-2}{qr}} \bigg\} \label{laga},
%\end{align}
%\textcolor{red}{should it be
\begin{align}
\|\bA\|_{q \to r} \leq \min\bigg\{\max_{1\leq i \leq m}\|\ba_i\|_2^{\frac{r-2}{r}}\sigma_{\max}(\bA)^{\frac{2}{r}} n^{\frac{q-2}{2q }},  \max_{1\leq i \leq m} \|\ba_i\|_p^{\frac{r-2}{r}} \sigma_{\max}(\bA)^{\frac{2}{r}} n^{\frac{q-2}{qr}} \bigg\} \label{laga}
\end{align}
where $\sigma_{\max}(\bA)$ denotes the largest singular value of $\bA$.
As a corollary, for any orthonormal matrix $\bA$ such that there exists a row with all elements in $\{-1,+1\}/\sqrt{n}$. Then, it holds that
\begin{align*}
\|\bA\|_{q \to r} = n^{\frac{(q-2)}{2q}}.
\end{align*}
Note that the Hadamard matrix belongs to this class.
\end{theorem} 
\begin{proof}
Observe that
\begin{align}
\|\bA\|_{q \to r}^r &= \sup_{\bx: \|\bx\|_q \leq 1} \sum_{i=1}^m |\ba_i^T \bx|^r\notag\\
& \leq  \sup_{\bx: \|\bx\|_q \leq 1} \max_{1\leq i \leq m} |\ba_i \bx|^{r-2} \bx^T \big(\sum_{i=1}^m \ba_i \ba_i^T\big) \bx\notag\\
&= \sup_{\bx: \|\bx\|_q \leq 1} \max_{1\leq i \leq m} |\ba_i \bx|^{r-2} \bx^T \bA^T \bA \bx \notag\\
&\leq  \sup_{\bx: \|\bx\|_q \leq 1} \max_{1\leq i \leq m} |\ba_i \bx|^{r-2} \sigma_{\max}(\bA)^2 \|\bx\|^2\notag \\
&\leq  \sup_{\bx: \|\bx\|_q \leq 1} \max_{1\leq i \leq m} \|\ba_i\|_2^{r-2} \|\bx\|_2^{r-2}\sigma_{\max}(\bA)^2 \|\bx\|_2^{2}\notag\\
&=  \sup_{\bx: \|\bx\|_q \leq 1} \max_{1\leq i \leq m}  \|\ba_i\|_2^{r-2}\sigma_{\max}(\bA)^2 \|\bx\|_2^r  \label{talet1}.
\end{align}
Now, for $\|\bx\|_q \leq 1$ by H\"{o}lder's inequality, we have
\begin{align*}
\|\bx\|_2^2 \leq n^{\frac{q-2}{q}} \|\bx\|_q^2 \leq n^{\frac{q-2}{q}}.
\end{align*}
Hence, we have
\begin{align}
\|\bA\|_{q \to r}^r  \leq  \max_{1\leq i \leq m} \|\ba_i\|_2^{r-2} \sigma_{\max}(\bA)^2 n^{\frac{(q-2)r}{2q}} \label{anet1}.
\end{align}
The inequality holds if $\ba_{i_0}= 1/\sqrt{n} \begin{bmatrix} \tau_1&\tau_2& \cdots & \tau_n\end{bmatrix}^T$ for some sequence $\{\tau_i\}_{i=1}^n \in \{-1,+1\}^n$ and $\bx=n^{-1/q} \begin{bmatrix} \tau_1&\tau_2& \cdots & \tau_n\end{bmatrix}^T$ for some $i_0$ and $\langle \ba_j, \ba_{i_0} \rangle =0$ for all $j \neq i_0$. 

On the other hand, we also have
\begin{align}
 \|\bA\|_{q \to r}^r &= \sup_{\bx: \|\bx\|_q \leq 1} \max_{1\leq i \leq m} |\ba_i \bx|^{r-2} \sigma_{\max}(\bA)^2 \|\bx\|_2^2\notag\\
 &=\sup_{\bx: \|\bx\|_q \leq 1} \max_{1\leq i \leq m} \|\ba_i\|_p^{r-2}  \|\bx\|_q^{r-2} \sigma_{\max}(\bA)^2 \|\bx\|_2^2\notag\\
 &\leq \max_{1\leq i \leq m} \|\ba_i\|_p^{r-2} \sigma_{\max}(\bA)^2 n^{\frac{q-2}{q}} \label{anet2}.
\end{align}
By combining \eqref{anet1} and \eqref{anet2}, we obtain \eqref{laga}. 
\end{proof}

\subsection{SVD decomposition}
In this section, we study the induced norm for a class of matrices via its singular value decomposition (SVD).
\begin{theorem} 
\label{thm: SVD decomposition}
Let $q \geq 2$ and $\tau_1, \tau_2, \cdots, \tau_n$ be an arbitrary tuple in $\{-1,+1\}^n$. Assume that $\bA \in \bbR^{m \times n}$ with a SVD decomposition:
\begin{align*} 
\bA= \bV \bSigma \bU^T,
\end{align*} which satisfies
\begin{align*}
\bu_1= \begin{bmatrix} \frac{\tau_1}{\sqrt{n}}, \frac{\tau_2}{\sqrt{n}}, \cdots, \frac{\tau_n}{\sqrt{n}} \end{bmatrix}^T \in \bbR^n,\quad
\bv_1=\begin{bmatrix} 1\\0 \\ \vdots \\ 0\end{bmatrix},\quad
\sigma_1 = \max_{1\leq i \leq \min\{m,n\}} \sigma_i,
\end{align*} where $\bu_1,\bv_1$ are respectively the first column of $\bU$ and  $\bV$, and $\sigma_1$ is the element with largest value in $\bSigma$ (note that $\min_{i \in \min\{m,n\}} \sigma_i \geq 0$).  Then, we have
\begin{align*}
\|\bA\|_{q \to r} =\sigma_1 n^{\frac{q-2}{2q}},
\end{align*} 
which is achieved at 
\begin{align}
\bx_*=n^{-1/q} (\tau_1,\tau_2, \cdots, \tau_n)^T \in \bbR^n. 
\end{align}
In particular, if $\bU$ is an orthonormal matrix with a column $\begin{bmatrix} \frac{\tau_1}{\sqrt{n}}, \frac{\tau_2}{\sqrt{n}}, \cdots, \frac{\tau_n}{\sqrt{n}} \end{bmatrix}^T$ for any tuple $(\tau_1, \tau_2, \cdots, \tau_n) \in \{-1,+1\}^n$ (such as the Hadamard matrix), then it holds that
\begin{align*}
\|\bU\|_{q \to r}= n^{\frac{q-2}{2q}}. 
\end{align*}
\end{theorem}

\begin{proof}
For any $\bx \in \bbR^n$, we have
\begin{align*}
\|\bA \bx\|_r^r =\sum_{i=1}^m |\ba_i^T \bx|^r
\geq &\max_{1 \leq i \leq m} |\ba_i^T \bx|^r.
\end{align*}
It follows that
\begin{align}
\max_{1\leq i \leq m} |\ba_i^T \bx| \leq \|\bA \bx\|_r, \qquad \forall \bx \in \bbR^n \label{c1-1}. 
\end{align}
Then, we have
\begin{align}
\|\bA \bx\|_r^r&= \sum_{i=1}^m |\ba_i^T \bx|^r\notag\\
&\leq \max_{1 \leq i \leq m} |\ba_i^T \bx|^{r-2} \sum_{j=1}^m |\ba_i^T \bx|^2\notag\\
&\leq \|\bA \bx\|_r^{r-2}  \sum_{j=1}^m |\ba_i^T \bx|^2\notag\\
&=  \|\bA \bx\|_r^{r-2}  \bx^T \bA^T \bA \bx \label{c2-1}.
\end{align}
The equality in \eqref{c2-1} happens if
\begin{align}
\bA \bx = \gamma \begin{bmatrix} 1\\0 \\ \vdots \\ 0\end{bmatrix} \in \bbR^m
\end{align} for some $\gamma \in \bbR$. 

From \eqref{c2-1} we obtain
\begin{align}
\|\bA \bx\|_r^2 &\leq \bx^T \bA^T \bA \bx \qquad \forall \bx \in \bbR^n \label{c3}\\
&\leq \sigma_{\max}^2(\bA) \|\bx\|_2^2 \label{c4},
\end{align} where $\sigma_{\max}(\bA)$ is the largest singular value of $\bA$. 

Now, assume that
\begin{align*}
\bA = \bV \Sigma \bU^T
\end{align*} where $\Sigma \in \bbR^{m \times n}$ with singular values $(\sigma_1,\sigma_2, \cdots, \sigma_{\min\{m,n\}})$ on the diagonal and $\sigma_1 \geq \sigma_2 \geq \cdots \geq \sigma_{\min\{m,n\}}\geq 0$. Then, we have
\begin{align*}
\bA^T \bA = \bU \Sigma^2 \bU^T
\end{align*}
Let $\bu_1, \bu_2, \cdots, \bu_n$ are columns of $\bU$. Then, the equality in \eqref{c4} happens if
\begin{align}
\bx = \beta \bu_1 \label{mofat2}. 
\end{align}

Finally, for any $q \geq 2$ we have
\begin{align}
\|\bx\|_2^2&=\sum_{j=1}^n x_j^2 \leq n^{\frac{q-2}{q}}\|\bx\|_q^2 \label{mofat3a} \\
& \leq n^{\frac{q-2}{q}} \label{mofat3}.
\end{align}
The equality in \eqref{mofat3a} and \eqref{mofat3} simultaneously happens if
\begin{align*}
\bx = n^{-1/q} (\tau_1,\tau_2, \cdots, \tau_n)^T \in \bbR^n,
\end{align*} where $\tau_i \in \{-1,+1\}$ for all $i \in [n]$. 

From \eqref{c4} and \eqref{mofat3}, we obtain
\begin{align*}
\|\bA\|_{q\to r}^2 &=\max_{\bx: \|\bx\|_q \leq 1} \|\bA \bx\|_r^2 \\
&\leq  \sigma_{\max}^2(\bA) n^{\frac{q-2}{q}},
\end{align*}
or
\begin{align}
\|\bA\|_{q\to r} &\leq \sigma_{\max}(\bA)n^{\frac{q-2}{2q}} \label{amet}.
\end{align}
The equality in \eqref{amet} happens if the following equations hold:
\begin{align}
\bA \bx &= \gamma \begin{bmatrix} 1\\0 \\ \vdots \\ 0\end{bmatrix} \in \bbR^m, \quad \mbox{for some} \quad \gamma \in \bbR, \notag\\
\bx &= \beta \bu_1,\notag\\
\bx &= n^{-1/q} (\tau_1,\tau_2, \cdots, \tau_n)^T \in \bbR^n,\label{met}
\end{align}  where $\tau_i \in \{-1,+1\}$ for all $i \in [n]$. 

Note that if $\bx=\beta \bu_1$,  we have
\begin{align*}
\bA \bx&= \beta \bA \bu_1\\
&=\beta \bV \bSigma \bU^T \bu_1 \\
&=\beta \bV \bSigma  \begin{bmatrix} 1\\0 \\ \vdots \\ 0\end{bmatrix} \\
&= \beta \sigma_{\max}(\bA) \bv_1,
\end{align*} where $\bv_1, \bv_2, \cdots, \bv_m$ are columns of $\bV$. 
Hence, all the equations in \eqref{met} hold if
\begin{align*}
\bu_1&= \begin{bmatrix} \frac{\tau_1}{\sqrt{n}}, \frac{\tau_2}{\sqrt{n}}, \cdots, \frac{\tau_n}{\sqrt{n}} \end{bmatrix}^T \in \bbR^n,\\
\bv_1&=\begin{bmatrix} 1\\0 \\ \vdots \\ 0\end{bmatrix},
\end{align*} where $\tau_i \in \{-1,+1\}$ for all $i \in [n]$.

In summary, we have
\begin{align*}
\|\bA\|_{q\to r}= \sigma_{\max}(\bA)n^{\frac{q-2}{2q}},
\end{align*}
if the SVD decomposition of $\bA$, say
\begin{align*} 
\bA= \bV \bSigma \bU^T
\end{align*} satisfies
\begin{align*}
\bu_1&= \begin{bmatrix} \frac{\tau_1}{\sqrt{n}}, \frac{\tau_2}{\sqrt{n}}, \cdots, \frac{\tau_n}{\sqrt{n}} \end{bmatrix}^T \in \bbR^n,\\
\bv_1&=\begin{bmatrix} 1\\0 \\ \vdots \\ 0\end{bmatrix},\\
\sigma_1 &= \max_{1\leq i \leq \min\{m,n\}} \sigma_i.
\end{align*}
\end{proof}
\subsection{Shear matrices and extensions}
In this section, we compute the induced norm of a shear matrix and its extensions.
\begin{theorem}
\label{thm:rank1shear} For any $\gamma \in \bbR$,
let $\bA=I+\gamma e_1e_2^T\in \bbR^{n \times n}$ be the shear matrix:
\begin{align*}
\bA = \begin{bmatrix} 1& \gamma& 0 & 0 & \cdots& 0&0 \\ 0&1& 0 & 0 & \cdots& 0&0\\ 0&0& 1 & 0 & \cdots& 0&0\\  0&0& 0 & 1 & \cdots& 0&0\\   \vdots&\vdots& \vdots & \vdots & \ddots& \vdots &\vdots\\0&0& 0 & 0 & \cdots& 1&0\\0&0& 0 & 0 & \cdots& 0&1
\end{bmatrix}  \in \bbR^{n \times n}.
\end{align*}
Let $\lambda_0$ be the unique positive solution of the following equation:
\begin{align}
(1+ \lambda^p)^{\frac{q}{p}}-1 =(1+\lambda^p)^{\frac{q}{p}} \bigg|\frac{\gamma}{\lambda}\bigg|^q \label{alo1},
\end{align} 
where $p,q \geq 1$ be such that $
\frac{1}{q}+\frac{1}{p}=1$. Then we have
\begin{align*}
\|\bA\|_{q \to q} = (1+\lambda_0^p)^{\frac{1}{p}} \bigg|\frac{\gamma}{\lambda_0}\bigg|.
\end{align*}
\end{theorem}
\begin{remark}
The equation \eqref{alo1} has a unique positive solution $\lambda$. Indeed, \eqref{alo1} is equivalent to
\begin{align*}
f(\lambda):=(1+\lambda^p)^{\frac{q}{p}}\bigg(1-\frac{|\gamma|^q}{\lambda^q}\bigg)=1.
\end{align*}
Direct computations yield
\[
f'(\lambda)
= q(1+\lambda^p)^{\frac{q}{p}-1}
\left(\lambda^{p-1} + |\gamma|^q \lambda^{-q-1}\right)>0\quad\text{for}\quad \lambda>0.
\]
Hence $f(\lambda)$ is increasing for $\lambda \in (0, \infty)$. Furthermore, $f(\lambda) \to -\infty $ as $\lambda \to 0^+$ and $f(\lambda) \to +\infty$ as $\lambda \to +\infty$. 
\end{remark}
\begin{proof}[Proof of Theorem 
\ref{thm:rank1shear}]
We have $\bA \bx=(x_1+\gamma x_2,x_2,\ldots, x_n)^T$. Hence we have
\begin{align}
\|\bA\|_{q\to q}^q&= \sup_{\bx: \|\bx\|_q \leq 1} \|\bA \bx\|_q^q\notag\\
&=\sup_{\bx: \|\bx\|_q \leq 1}  |x_1+\gamma x_2|^q+ \sum_{i=2}^n  |x_i|^q\notag\\
&\leq \sup_{\bx: \|\bx\|_q \leq 1}  |x_1+\gamma x_2|^q+1-|x_1|^q\notag\\
&\leq \sup_{x_1,x_2: |x_1|^q+|x_2|^q \leq 1}  |x_1+\gamma x_2|^q+1-|x_1|^q \label{t1}.
\end{align}
Let $p, q \geq 1$ be such that
\[
\frac{1}{q}+\frac{1}{p}=1.
\]
Applying the following elementary inequality (which is due to convexity)
\[
|a+b|^q\leq \frac{|a|^q}{(1-\alpha)^{q-1}}+\frac{|b|^q}{\alpha^{q-1}}\quad \text{for all}\quad q\geq 1, \alpha\in(0,1),
\]
to $a=x_1,~ b=\gamma x_2= \lambda_0 \frac{\gamma}{\lambda_0} x_2$ and $\alpha=\frac{\lambda_0^p}{1+\lambda_0^p}$, the first term in the right-hand side of \eqref{t1} can be estimated by
\begin{align}
|x_1+\gamma x_2|^q&=\bigg|x_1+\lambda_0 \frac{\gamma}{\lambda_0} x_2\bigg|^q\notag\\
&\leq \frac{|x_1|^q}{\Big(\frac{1}{1+\lambda_0^p}\Big)^{q-1}}+\frac{\lambda_0^q}{\Big(\frac{\lambda_0^p}{1+\lambda_0^p}\Big)^{q-1}}\bigg| \frac{\gamma}{\lambda_0}\bigg|^q |x_2|^q\notag
\\&= (1+  \lambda_0^p)^{\frac{q}{p}}\bigg[|x_1|^q+ \bigg| \frac{\gamma}{\lambda_0}\bigg|^q |x_2|^q\bigg] \label{t2},
\end{align}
where to obtain the last equality we have used the fact that $p(q-1)=q$. From \eqref{t1} and \eqref{t2} we obtain
\begin{align*}
\|\bA\|_{q\to q}^q &\leq \sup_{x_1,x_2: |x_1|^q+|x_2|^q \leq 1}  (1+  \lambda_0^p)^{\frac{q}{p}}\bigg[|x_1|^q+ \bigg| \frac{\gamma}{\lambda_0}\bigg|^q |x_2|^q\bigg] +1-|x_1|^q\\
&= \sup_{x_1,x_2: |x_1|^q+|x_2|^q \leq 1} \bigg[(1+  \lambda_0^p)^{\frac{q}{p}}-1\bigg]|x_1|^q+ (1+  \lambda_0^p)^{\frac{q}{p}}\bigg| \frac{\gamma}{\lambda_0}\bigg|^q |x_2|^q\\
&=(1+  \lambda_0^p)^{\frac{q}{p}}\bigg| \frac{\gamma}{\lambda_0}\bigg|^q (|x_1|^q+ |x_2|^q)\\
&\leq (1+  \lambda_0^p)^{\frac{q}{p}} \bigg|\frac{\gamma}{\lambda_0}\bigg|^q,
\end{align*}
where we have used the definition of $\lambda_0$ to obtain the second equality above. Hence
\begin{align}
\|\bA\|_{q\to q}\leq (1+  \lambda_0^p)^{\frac{1}{p}} \bigg|\frac{\gamma}{\lambda_0}\bigg| \label{amen}.
\end{align}
The equality in \eqref{amen} happens if
\begin{align*}
x_1&= \bigg(1+\lambda_0^p \bigg|\frac{\lambda_0}{\gamma}\bigg|^q\bigg)^{-\frac{1}{q}},\\
x_2&=\mbox{sgn}(\gamma)\bigg(1+\lambda_0^p \bigg|\frac{\lambda_0}{\gamma}\bigg|^q\bigg)^{-\frac{1}{q}} \bigg(\lambda_0^p \bigg|\frac{\lambda_0}{\gamma}\bigg|^q\bigg)^{1/q},\\
x_i&=0, \qquad \forall i \in [3,n]. 
\end{align*}
\end{proof}
Combining Theorem \ref{thm: SVD decomposition} and Theorem \ref{thm:rank1shear} we obtain the following result.
\begin{theorem}
\label{thm: SVDrank1shear}
Let $\bA \in \bbR^{n \times n}$ be an invertible matrix such that
\begin{align*}
\bA =\bB \bC,
\end{align*}
where $\bB$ and $\bC$ satisfy the following assumptions.

(i) $\bB$ has a SVD decomposition $\bB= \bV \bSigma \bU^T  \in \bbR^{n \times n}$,
which satisfies
\begin{align*}
\bu_1= \begin{bmatrix} \frac{1}{\sqrt{n}}, \frac{1}{\sqrt{n}}, \cdots, \frac{1}{\sqrt{n}} \end{bmatrix}^T \in \bbR^n,\quad
\bv_1=\begin{bmatrix} 1\\0 \\ \vdots \\ 0\end{bmatrix},\quad
\sigma_1 = \max_{1\leq i \leq \min\{m,n\}} \sigma_i,
\end{align*} where $\bu_1,\bv_1$ are first columns of $\bU, \bV$, and $\sigma_1$ is the element with largest value in $\bSigma$ (note that $\min_{i \in n} \sigma_i > 0$).

(ii) $\bC \in \bbR^{n \times n}=I+\gamma e_1^T e_2$ is the shear matrix
\begin{align*}
\bC:=\begin{bmatrix} 1& \gamma& 0 & 0 & \cdots& 0&0 \\ 0&1& 0 & 0 & \cdots& 0&0\\ 0&0& 1 & 0 & \cdots& 0&0\\  0&0& 0 & 1 & \cdots& 0&0\\   \vdots&\vdots& \vdots & \vdots & \ddots& \vdots &\vdots\\0&0& 0 & 0 & \cdots& 1&0\\0&0& 0 & 0 & \cdots& 0&1
\end{bmatrix},  
\end{align*} 
where $\gamma \in \bbR_+$ satisfying:
\begin{align*}
\gamma=\frac{\bigg[\bigg(1+\frac{n}{|1-\gamma|^q+n-1}\bigg)^{p/q}-1\bigg]^{1/p}}{\bigg(1+\frac{n}{|1-\gamma|^q+n-1}\bigg)^{1/q}}. 
\end{align*}
Then, we have
\begin{align*}
\|\bA\|_{q\to r}=n^{\frac{1}{q}}\bigg[|1-\gamma|^q + (n-1)\bigg]^{-\frac{1}{q}}  \sigma_1 n^{\frac{q-2}{2q}},
\end{align*} which achieves at
\begin{align*}
\bx^*= \frac{1}{\xi} n^{-\frac{1}{q}} \bA^{-1}\bB \begin{bmatrix} 1\\ 1 \\ \vdots \\ 1\end{bmatrix}=\frac{1}{\xi}  \bC \by^*. 
\end{align*}
\end{theorem}
\begin{proof}
%Observe that  $\bB$ is invertible and according to Theorem \ref{thm: SVD decomposition} we have
%\begin{align*}
%\|\bB\|_{q \to r} =\sigma_1 n^{\frac{q-2}{2q}},
%\end{align*} which happens at the optimizer
%\begin{align*}
%\bx_*=n^{-1/q} (1,1, \cdots, 1)^T \in \bbR^n. 
%\end{align*}
For any $\xi>0$, we have
\begin{align}
\|\bA\|_{q\to r}&=\sup_{\bx: \|\bx\|_q \leq 1}\|\bA \bx\|_r\notag \\
&=\sup_{\bx: \|\bx\|_q \leq 1} \big\| \bB (\bB^{-1} \bA \bx)\big\|_r\notag\\
&= \sup_{\|\bA^{-1} \bB \by\|_q \leq 1} \big\|\bB \by\big\|_r \notag\\
&=  \frac{1}{\xi} \sup_{\|\bA^{-1} \bB \by\|_q \leq \xi}\big\|\bB \by\big\|_r \label{k1}.
\end{align}
Now suppose that 
\begin{align*}
\|\bC\|_{q\to q} \leq \frac{1}{\xi}. 
\end{align*}
Let $\by\in \bbR^n$ satisfy $\|\bA^{-1} \bB \by\|_q \leq \xi$. Then we have
\begin{align}
\|\by\|_q&= \big\| \bB^{-1} \bA (\bA^{-1} \bB \by)\big\|_q\notag\\
& = \big\| \bC (\bA^{-1} \bB \by)\big\|_q\notag\\
&\leq \|\bC\|_{q\to q} \|\bA^{-1} \bB \by\|_q\notag\\
&\leq \frac{1}{\xi} \xi =1,\label{k2-1}
\end{align} 
where we have used $\bC=\bB^{-1}\bA$ to obtain the second equality above. Hence, from \eqref{k1} and \eqref{k2-1} we obtain
\begin{align*}
\|\bA\|_{q \to r} &\leq \frac{1}{\xi} \sup_{\by\in \bbR^n: \|\by\|_q \leq 1} \|\bB \by\|_r \\
&=\frac{1}{\xi}  \|\bB\|_{q \to r}\\
&= \frac{1}{\xi}  \sigma_1 n^{\frac{q-2}{2q}}.
\end{align*}
Note that $\|\bB\|_{q \to r}$ achieves the optimal value $\sigma_1 n^{\frac{q-2}{2q}}$ at $\by_*=n^{-1/q} (1,1, \cdots, 1)^T \in \bbR^n$ (by Theorem \ref{thm: SVD decomposition}). Hence, 
$\sup_{\|\bA^{-1} \bB \by\|_q \leq \xi} \big\|\bB \by\big\|_r =\|\bB\|_{q \to r}= \sigma_1 n^{\frac{q-2}{2q}} $ if
\begin{align*}
\|\bA^{-1} \bB \by^*\|_q \leq \xi,
\end{align*}
or equivalently
\begin{align*}
\|\bC^{-1} \by^*\|_q \leq \xi.
\end{align*}
In summary, we have proved that
\begin{align}
\|\bA\|_{q\to r} = \frac{1}{\xi} \sigma_1 n^{\frac{q-2}{2q}},
\end{align} 
for some $\xi>0$ under the condition that there exists $\bC \in \bbR^{n \times n}$, invertible, such that
\begin{align}
&\|\bC\|_{q\to q} \leq \frac{1}{\xi}, \label{c1} \\
&\|\bC^{-1} \by^*\|_q \leq \xi \label{c2},
\end{align} where $\by_*=n^{-1/q} (1,1, \cdots, 1)^T \in \bbR^n$. 

The existence of $\bC$ which satisfies both \eqref{c1} and \eqref{c2} is guaranteed, for example $\bC=\bI$ in which $\xi=1$ (see Lemma \ref{lem: diagonal}). Now, we construct a nontrivial class of matrices $\bC$ which satisfies these conditions. Indeed, let
\begin{align*}
\bC:=\begin{bmatrix} 1& \gamma& 0 & 0 & \cdots& 0&0 \\ 0&1& 0 & 0 & \cdots& 0&0\\ 0&0& 1 & 0 & \cdots& 0&0\\  0&0& 0 & 1 & \cdots& 0&0\\   \vdots&\vdots& \vdots & \vdots & \ddots& \vdots &\vdots\\0&0& 0 & 0 & \cdots& 1&0\\0&0& 0 & 0 & \cdots& 0&1
\end{bmatrix},  
\end{align*} for some $\gamma \in \bbR$. Then, by Theorem \ref{thm:rank1shear}, we have
\begin{align*}
\|\bC\|_{q\to q} = (1+\lambda_0^p)^{\frac{1}{p}} \bigg|\frac{\gamma}{\lambda_0}\bigg|,
\end{align*}
where $\lambda_0$ is the unique solution of the following equation:
\begin{align}
(1+ \lambda^p)^{\frac{q}{p}}-1 =(1+\lambda^p)^{\frac{q}{p}} \bigg|\frac{\gamma}{\lambda}\bigg|^q \label{aloo}.
\end{align} 
For this matrix $\bC$, we have
\begin{align*}
\bC^{-1}=\begin{bmatrix} 1& -\gamma& 0 & 0 & \cdots& 0&0 \\ 0&1& 0 & 0 & \cdots& 0&0\\ 0&0& 1 & 0 & \cdots& 0&0\\  0&0& 0 & 1 & \cdots& 0&0\\   \vdots&\vdots& \vdots & \vdots & \ddots& \vdots &\vdots\\0&0& 0 & 0 & \cdots& 1&0\\0&0& 0 & 0 & \cdots& 0&1
\end{bmatrix}=(I-\gamma e_1^T e_2),
\end{align*}
so, by taking $\by^*=n^{-1/q} (1,\ldots, 1)^T$, we have
\begin{align}
\big\|\bC^{-1} \by^*\big\|_q &= n^{-\frac{1}{q}} \bigg\|\begin{bmatrix} 1& -\gamma& 0 & 0 & \cdots& 0&0 \\ 0&1& 0 & 0 & \cdots& 0&0\\ 0&0& 1 & 0 & \cdots& 0&0\\  0&0& 0 & 1 & \cdots& 0&0\\   \vdots&\vdots& \vdots & \vdots & \ddots& \vdots &\vdots\\0&0& 0 & 0 & \cdots& 1&0\\0&0& 0 & 0 & \cdots& 0&1
\end{bmatrix}\begin{bmatrix} 1\\1\\ 1\\1\\ \vdots \\1\\1\end{bmatrix}\bigg\|_q\notag\\
&= n^{-\frac{1}{q}}\bigg[|1-\gamma|^q + (n-1)\bigg]^{\frac{1}{q}} \label{atom2}.
\end{align}
Now, define
\begin{align}
\xi:=n^{-\frac{1}{q}}\bigg[|1-\gamma|^q + (n-1)\bigg]^{\frac{1}{q}} \label{alo0}.
\end{align}
Then, we need to show that there exists $\gamma \in \bbR$ such that
\begin{align}
(1+\lambda_0^p)^{\frac{1}{p}} \bigg|\frac{\gamma}{\lambda_0}\bigg|=\frac{1}{\xi} \label{alo2}
\end{align}
and
\begin{align}
(1+ \lambda_0^p)^{\frac{q}{p}}-1 =(1+\lambda_0^p)^{\frac{q}{p}} \bigg|\frac{\gamma}{\lambda_0}\bigg|^q \label{alo}.
\end{align} 
Indeed, from \eqref{alo2} and \eqref{alo} we have
\begin{align*}
(1+ \lambda_0^p)^{\frac{q}{p}}-1 =\frac{1}{\xi^q},
\end{align*} or equivalently
\begin{align*}
\lambda_0=\bigg[\bigg( \frac{\xi^q+1}{\xi^q}\bigg)^{\frac{p}{q}}-1\bigg]^{\frac{1}{p}}.
\end{align*}
To satisfy \eqref{alo2}, we choose 
\begin{align*}
|\gamma|=\frac{\lambda_0}{(1+\lambda_0^p)^{\frac{1}{p}}} = \frac{\bigg[\bigg( \frac{\xi^q+1}{\xi^q}\bigg)^{\frac{p}{q}}-1\bigg]^{\frac{1}{p}}}{\bigg( \frac{\xi^q+1}{\xi^q}\bigg)^{\frac{1}{q}}},
\end{align*} or equivalently
\begin{align}
g(\gamma)=|\gamma|-\frac{\bigg[\bigg(1+\frac{n}{|1-\gamma|^q+n-1}\bigg)^{p/q}-1\bigg]^{1/p}}{\bigg(1+\frac{n}{|1-\gamma|^q+n-1}\bigg)^{1/q}}=0 \label{balo}.
\end{align}
Now, observe that
\begin{align*}
g(0)&=-\frac{\bigg[\bigg(1+\frac{n}{n}\bigg)^{p/q}-1\bigg]^{1/p}}{\bigg(1+\frac{n}{n}\bigg)^{1/q}}<0,\\
\lim_{\gamma \to \infty} g(\gamma)&=+\infty.  
\end{align*}
Hence, there exists $\gamma_0>0$ which is a solution of \eqref{alo2}. 

In summary, we have proved that
\begin{align*}
\|\bA\|_{q\to r} &= \frac{1}{\xi} \sigma_1 n^{\frac{q-2}{2q}}\\
&=n^{\frac{1}{q}}\bigg[|1-\gamma|^q + (n-1)\bigg]^{-\frac{1}{q}}  \sigma_1 n^{\frac{q-2}{2q}},
\end{align*} which achieves at
\begin{align*}
\bx^*= \frac{1}{\xi} n^{-\frac{1}{q}} \bA^{-1}\bB \begin{bmatrix} 1\\ 1 \\ \vdots \\ 1\end{bmatrix}=\frac{1}{\xi}  \bC^{-1} \by^*. 
\end{align*}
Note that
\begin{align*}
\|\bx^*\|_q=\frac{1}{\xi} \big\|\bC^{-1} \by^*\big\|_q=1. 
\end{align*}
\end{proof}
\subsection{Toeplitz matrices}
In this section we calcualte the induced norm of a class of matrices that cover bi-diagonal Toeplitz matrices with identity entries.
\begin{theorem} 
\label{thm: Toeplitz}
Let $\bA\in \bbR^{n \times n}$ be a class of matrices such that
\[
(\bA x)_i=x_{i_1}+\ldots+ x_{i_k},\quad i=1,\ldots, n
\]
for some $1\leq k\leq n$ such that the collection $\{x_{i1},\ldots,x_{i_k}\}_{i=1}^n$consists of $k$ copies of the set $\{x_1,\ldots, x_n\}$. Then
\[
\|\bA\|_{q \to q}= k.
\]
\end{theorem}
\begin{proof} 
We have
\begin{align*}
\|\bA x\|^q_{q}&=\sum_{i=1}^n |(Ax)_i|^q  \notag  \\
&=\sum_{i=1}^n |x_{i_1}+\ldots+x_{i_k}|^q\notag
\\&\leq k^{q-1}\sum_{i=1}^n(|x_{i_1}|^q+\ldots+|x_{i_k}|^q)
\\&=k^{q}\sum_{i=1}^n |x_i|^q=k^q\|x\|_q^q,
\end{align*}
where the first inequality is an application of Hölder's inequality and the third equality follows from the assumption that the collection $\{x_{i1},\ldots,x_{i_k}\}_{i=1}^n$consists of $k$ copies of the set $\{x_1,\ldots, x_n\}$.

Now, for $\bx^*=(x_1^*,x_2^*,\cdots, x_n^*)$ where $x_1^*=x_2^*=\cdots = x_n^*=n^{-\frac{1}{q}}$, we have 
$\|\bx^*\|_q=1$ and
$(\bA \bx^*)_i=k n^{-\frac{1}{q}}$ for all $i=1,\ldots, n$, from which we have
\begin{align*}
\|\bA \bx^*\|_q^q=\sum_{i=1}^n |(Ax)_i|^q=n (k n^{-\frac{1}{q}})^q=k^q.
\end{align*} 
Hence, we obtain
\begin{align*}
\|\bA\|_{q \to q}= k.
\end{align*}
\end{proof}
%\textcolor{blue}{can we apply the trick of Theorem \ref{thm: SVDrank1shear} for this matrix? I don't know for sure. It seems to me that the technique in \ref{thm: SVDrank1shear} is hard to apply for Toepliz matrices.}
\begin{remark}    
As a corollary of Theorem \ref{thm: Toeplitz}, let $\bA\in \bbR^{n \times n}$ be a Toeplitz matrix with the following structure:
\begin{align*}
\bA=\begin{bmatrix} 1&1& 0&0& \cdots  &0 &0 &0 \\ 0&1&1&0& \cdots& 0&0&0 \\  \vdots& \vdots&\vdots& \vdots& \ddots& \vdots& \vdots& \vdots \\ 0&0& 0&0& \cdots  &1 &1 &0 \\  0&0& 0&0& \cdots  &0 &1 &1\\  1&0& 0&0& \cdots  &0 &0 &1     \end{bmatrix}. 
\end{align*}
Then, it holds that
\begin{align*}
\|\bA\|_{q \to q}= 2. 
\end{align*}
Similarly, let
\begin{align*}
\bA=\begin{bmatrix} 1&-1& 0&0& \cdots  &0 &0 &0 \\ 0&1&-1&0& \cdots& 0&0&0 \\  \vdots& \vdots&\vdots& \vdots& \ddots& \vdots& \vdots& \vdots \\ 0&0& 0&0& \cdots  &1 &-1 &0 \\  0&0& 0&0& \cdots  &0 &1 &-1\\  -1&0& 0&0& \cdots  &0 &0 &1     \end{bmatrix} \end{align*}
Suppose that $n$ is odd. Then it holds that $\|\bA\|_{q\to q}=2$, which is achieved at the optimizer $x_i^*=(-1)^i n^{-1/q}$.  
\end{remark}
\subsection{Orthogonal transformations}
In this section, we study the induced norm $\|\cdot\|_{q\to r}$ in relation to orthogonal transformations.
\begin{theorem}
\label{thm: Orthogonal}
Let $r\geq 2$ and $q\geq 2$. Then, for any orthonormal matrix $\bU \in \bbR^{n \times n}$ with the row vectors $\bu_1^T, \bu_2^T, \cdots, \bu_n^T$, there exists a diagonal matrix $\bLambda$ such that
\begin{align*}
\|\bU \bLambda\|_{q \to r}=1.
\end{align*}
More specifically, 
\begin{align*}
\bLambda=\mathrm{diag}(|u_{i,1}|^{\frac{q-2}{q}}, |u_{i,2}|^{\frac{q-2}{q}}, \cdots, |u_{i,n}|^{\frac{q-2}{q}})
\end{align*} for any $i \in [n]$. 

The optimizer is
\begin{align*}
x_j= |u_{i,j}|^{\frac{2}{q}}\mathrm{sgn}(u_{i,j}), \qquad \forall j \in [n]. 
\end{align*}
\end{theorem}

\begin{proof} Assume that $\bLambda=\mbox{diag}(\lambda_1, \lambda_2, \cdots, \lambda_n)$. Then, for any $\bx \in \bbR^n$ we have
\begin{align}
\|\bU \bLambda \bx\|_r^r&=\sum_{i=1}^n |\bu_i^T \bLambda \bx|^r \notag\\
&= \sum_{i=1}^n (|\bu_i^T \bLambda \bx|^2)^{r/2}\notag\\
&\leq \bigg(\sum_{i=1}^n |\bu_i^T \bLambda \bx|^2\bigg)^{r/2} \label{y1}.
\end{align}
The equality in \eqref{y1} happens if
\begin{align*}
\bU \bLambda \bx =\gamma \be_i
\end{align*}  for some $\gamma \in \bbR$, where $\be_i$ is the vector with all elements being zero except the $i$-th element being equal to 1. 

From \eqref{y1} we obtain
\begin{align}
\|\bU \bLambda \bx\|_r^2 &\leq \sum_{i=1}^n |\bu_i^T \bLambda \bx|^2 \notag\\
&= \bx^T \bLambda \sum_{i=1}^n \bu_i \bu_i^T \bLambda \bx\notag\\
&= \bx^T \bLambda \bU^T \bU \bLambda \bx\notag\\
&= \|\by\|_2^2 \label{y2},
\end{align}
where
\begin{align*}
\by^T=(\lambda_1 x_1, \lambda_2 x_2, \cdots, \lambda_n x_n). 
\end{align*}
Now, we have
\begin{align}
\|\by\|_2^2&= \sum_{i=1}^n x_i^2 \lambda_i^2\notag\\
&\leq \bigg(\sum_{i=1}^n |\lambda_i|^{\frac{2q}{q-2}}\bigg)^{\frac{q}{q-2}}\bigg(\sum_{i=1}^n |x_i|^q\bigg)^{\frac{2}{q}} \label{mage} \\
&\leq 1 \label{mage1}
\end{align} if we choose $\lambda_i$ such that
\begin{align}
\sum_{i=1}^n |\lambda_i|^{\frac{2q}{q-2}}=1 \label{m1}. 
\end{align}
The equalities in \eqref{mage} and \eqref{mage1} simultaneously happen if
\begin{align}
\sum_{i=1}^n |\lambda_i|^{\frac{2q}{q-2}}&=1 \label{z1}\\
|x_i|=|x_i^*|&=|\lambda_i|^{\frac{2}{q-2}}, \qquad \forall i \in [n].\label{z2}
\end{align}
From \eqref{y1}, \eqref{y2}, and \eqref{mage1} we finally obtain
\begin{align}
\|\bU \bLambda\|_{q \to r} &=\sup_{\bx: \|\bx\|_q \leq 1} \|\bU \bLambda \bx\|_r\notag\\
&\leq 1 \label{k2}.
\end{align}
The equality in \eqref{k2} happens if
\begin{align}
\bU \bLambda \bx_* &=\be_i, \label{k10} \\
\sum_{i=1}^n |\lambda_i|^{\frac{2q}{q-2}}&=1, \label{k12},\\
x_j&=|\lambda_j|^{\frac{2}{q-2}}\sgn(\lambda_j), \qquad \forall j \in [n] \label{k12b}.  
\end{align}
Note that \eqref{k10} is equivalent to
\begin{align*}
\bLambda \bx_* = \bU^T \be_i = \bu_i,
\end{align*}
or equivalently
\begin{align}
|\lambda_j|^{\frac{q}{q-2}} \mbox{sgn}(\lambda_j) \mbox{sgn}(x_j)=u_{i,j}, \qquad \forall j \in [n] \label{k11}.
\end{align}
Under the condition that \eqref{k11} holds then \eqref{k12} also holds since $\bU$ is an orthonormal matrix, so $\|\bu_i\|_2^2=1$. Finally, \eqref{k11} happens if
\begin{align*}
\lambda_j = |u_{i,j}|^{\frac{q-2}{q}},
\end{align*}
and
\begin{align*}
x_j= |\lambda_j|^{\frac{2}{q-2}} \mbox{sgn}(\lambda_j)= |u_{i,j}|^{\frac{2}{q}}\mbox{sgn}(u_{i,j}), \qquad \forall j \in [n]. 
\end{align*}
\end{proof}
\begin{theorem} 
\label{thm: orthogonalSVD}
Let $\bA \in \bbR^{m \times n}$ with the following decomposition:
\begin{align*}
\bA= \bU \bSigma \bV \bSigma_V,
\end{align*} where $\bU \in \bbR^{m \times n}$ is an arbitrary orthogonal matrix with the first column $(1,0,0,\cdots,0)^T \in \bbR^n$, $\bSigma \in \bbR^{m \times n}$ is a diagonal matrix, and $\bV$ is an arbitrary orthogonal matrix, $\bSigma_V= \mbox{diag}(|v_{11}|^{\frac{q-2}{q}},|v_{12}|^{\frac{q-2}{q}}, \cdots, |v_{1n}|^{\frac{q-2}{q}})$ where
\begin{align*}
\bV=\begin{bmatrix} \bv_1^T \\ \bv_2^T \\ \vdots \\ \bv_n^T \end{bmatrix}. 
\end{align*}
 Then, it holds that
\begin{align*}
\|\bA\|_{q \to r}= \lambda_{\max}(\bSigma),
\end{align*} where $\lambda_{\max}(\bSigma)$ is the largest absolute-value diagonal element of $\bSigma$.
\end{theorem}
\begin{proof}
Assume that $\bSigma_V=\mathrm{diag}(\lambda_1,\lambda_2,\cdots, \lambda_n)$ where $\lambda_i \in \bbR^+$ for all $i=1,\ldots, n$. Then, %\textcolor{red}{what are the $\{\lambda_i\}$? is $\bLambda_V$ or $\Sigma_V$?} 
we have
\begin{align}
 \|\bSigma_V \bx\|_2^2 &=\sum_{i=1}^n |\lambda_i|^2 |x_i|^2 \notag\\
 &\leq \bigg(\sum_{i=1}^n |v_{1i}|^2\bigg)^{\frac{q-2}{q}} \|\bx\|_q^2\notag\\
 &= \|\bx\|_q^2,  \qquad \forall \bx \in \bbR^n \label{G1}.
\end{align}
The equality in \eqref{G1} happens if
\begin{align}
\bx=   \begin{bmatrix} |v_{11}|^{\frac{2}{q}} \sgn(v_{11})\\  |v_{12}|^{\frac{2}{q}} \sgn(v_{12})\\ \vdots \\  |v_{1n}|^{\frac{2}{q}} \sgn(v_{1n}) \end{bmatrix} \label{H1}. 
\end{align}
It follows that
\begin{align}
\|\bA\|_{q \to 2}&=\sup_{\bx: \|\bx\|_q \leq 1} \|\bA \bx\|_2\notag\\
&\leq \sup_{\bx: \|\bSigma_V \bx\|_2 \leq 1} \|\bA \bx\|_2 \label{G2a} \\
&=  \sup_{\by: \|\by\|_2 \leq 1} \|\bA \bSigma_V^{-1}\by\|_2 \label{G2}\\
&=  \|\bA \bSigma_V^{-1}\|_{2 \to 2}\notag\\
&=  \|\bU \bSigma \bV\|_{2\to 2} \notag\\
&=  \lambda_{max}(\bSigma)
\label{G3}
\end{align}
The equality in \eqref{G3} happens at 
\begin{align*}
\by^*= \bv_1, 
\end{align*} 
or
\begin{align*}
\bx^* =  \bSigma_V^{-1} \by^* =     \begin{bmatrix} |v_{11}|^{\frac{2}{q}} \sgn(v_{11})\\  |v_{12}|^{\frac{2}{q}} \sgn(v_{12})\\ \vdots \\  |v_{1n}|^{\frac{2}{q}} \sgn(v_{1n}) \end{bmatrix},
\end{align*}
which satisfies \eqref{H1}. 

Hence, we have
\begin{align*}
\|\bA\|_{q \to 2} = \lambda_{\max}(\bSigma)
\end{align*} for any $\bU \in \bbR^{m \times m}, \bSigma \in \bbR^{m \times n}, \bV \in \bbR^{n \times n}$.

Now, for all $r \geq 2$ we have
\begin{align*}
\|\bA\|_{q \to r}^r&= \sup_{\bx: \|\bx\|_q \leq 1} \|\bA \bx\|_r^r\\
&= \sup_{\bx:\|\bx\|_q \leq 1} \sum_{i=1}^m |\ba_i^T \bx|^r\\
&\leq \sup_{\bx:\|\bx\|_q \leq 1} \bigg(\sum_{i=1}^m |\ba_i^T \bx|^2\bigg)^{\frac{r}{2}}\\
&\leq \sup_{\bx: \|\bx\|_q \leq 1} \|\bA \bx\|_2^r. 
\end{align*}
Hence, we have
\begin{align}
\|\bA\|_{q \to r} &\leq \sup_{\bx: \|\bx\|_q \leq 1} \|\bA \bx\|_2 \label{H3}\\
&= \|\bA \|_{q \to 2}\notag\\
&= \lambda_{\max}(\bSigma) \label{H4}. 
\end{align}
The equality in \eqref{H3} happens if
\begin{align}
\bA \bx = \gamma \begin{bmatrix} 1 \\ 0 \\ 0 \\ \vdots \\ 0 \end{bmatrix} \label{H0},
\end{align} for some $\gamma \in \bbR$. 

%The equality in \eqref{H4} happens at
%\begin{align}
%\bx^*=   \begin{bmatrix} |v_{11}|^{\frac{2}{q}} %\sgn(v_{11})\\  |v_{12}|^{\frac{2}{q}} %\sgn(v_{12})\\ \vdots \\  |v_{1n}|^{\frac{2}{q}} %\sgn(v_{1n}) \end{bmatrix} \label{H0b}. 
%\end{align}

Now, we show that at the optimizer $\bx=\bx^*$ then \eqref{H0} is satisfied when $\bu_1=(1,0,0,\cdots, 0)^T \in \bbR^m$. Indeed, we have
\begin{align*}
\bA \bx^*&= \bU \bSigma \bV \bSigma_V   \begin{bmatrix} |v_{11}|^{\frac{2}{q}} \sgn(v_{11})\\  |v_{12}|^{\frac{2}{q}} \sgn(v_{12})\\ \vdots \\  |v_{1n}|^{\frac{2}{q}} \sgn(v_{1n}) \end{bmatrix} \\
&= \bU \bSigma \bV \bv_1\\
&= \lambda_{\max}(\bSigma)\bu_1 \\
&= \lambda_{\max} \begin{bmatrix} 1 \\ 0 \\ 0 \\ \vdots \\ 0 \end{bmatrix}. 
\end{align*}

Hence, finally we have
\begin{align*}
\|\bA \|_{q \to r} = \lambda_{\max}(\bSigma). 
\end{align*}
\end{proof}
\subsection{The norm $1\rightarrow r$}
This section deals with the norm $\|\cdot\|_{1\rightarrow r}$.
\begin{theorem} 
\label{thm: 1rnorm}
Let $\bA \in \bbR^{m \times n}$ with columns $\bc_1,\bc_2, \cdots, \bc_n$. Let $j_0=\mbox{argmax}_{1\leq j \leq n} \|\bc_j\|_r$. Then, it holds that
\begin{align*}
\|\bA\|_{1 \to r} = \max_{1 \leq j \leq n} \|\bc_j\|_r. 
\end{align*}
\end{theorem}
\begin{proof} Observe that
\begin{align*}
\|\bA\|_{1 \to r} &=\max_{\bx: \|\bx\|_1 \leq 1} \| \bA \bx \|_r\\
&=\max_{\bx: \|\bx\|_1 \leq 1} \bigg\|\sum_{j=1}^n \bc_j x_j \bigg\|_r\\
&=\max_{\bx: \|\bx\|_1 \leq 1} \sum_{j=1}^n \|\bc_j \|_r |x_j |\\
&\leq \max_{\bx: \|\bx\|_1 \leq 1} \max_{1 \leq j \leq n}   \sum_{j=1}^n |x_j |\\
&= \max_{1 \leq j \leq n}  \|\bc_j\|_r   \max_{\bx: \|\bx\|_1 \leq 1} \sum_{j=1}^n |x_j |\\\
&\leq \max_{1 \leq j \leq n}  \|\bc_j\|_r. 
\end{align*}
On the other hand, let $\bx_*= \be_{j_0}$, where $\be_{j_0}$ is the vector with all elements zero except at the position $j_0$ at which it is equal to $1$. Then, we have
$ \|\bx_*\|_q=1$ and $\| \bA \bx_* \|_r = \|\bc_{j_0}\|_r$. 
\end{proof}
\section{Summary and outlook}
\label{sec:summary}
Determining explicit values of induced norms is important because they quantify how operators amplify inputs, underpinning stability, error analysis, and conditioning. Exact formulas yield sharper results and avoid conservative estimates, improving both theory and computation. However, apart from the classical cases $q,r\in
\{1,2,\infty\}$, essentially no purely analytic expressions for $\|\bA\|_{q\to r}$ are known. Existing work provides approximations or bounds rather than exact structural characterizations, largely due to the underlying maximization problem being highly non-convex with potentially many stationary points. In this paper we have calculated explicitly $\|\bA\|_{q\to r}$ for all $q,r\geq 1$ for several important classes of matrices. Our approach relies on thoroughly analyzing the extremal structure of maximizers and the geometry of the constraint sets. We expect that the present work paves the way for future investigations on this research direction for more complex classes of matrices and for extensions to infinite-dimensional Hilbert spaces.
\section*{Acknowledgment} Lan V. Truong would like to thank Ho Chi Minh City University of Technology (HCMUT), Vietnam National University Ho Chi Minh City (VNU-HCM), for supporting this study. The research of MHD was supported by EPSRC (Grant Number EP/Y008561/1). 

\bibliographystyle{tmlr}
%\bibliography{main}
%\bibliographystyle{iclr2024_conference}
\bibliography{isitbib}

\appendix
%\section{Appendix}

\end{document}